\documentclass[12pt]{amsart}

\usepackage{xcolor}
\usepackage[vcentermath,enableskew]{youngtab}
\usepackage{amssymb}
\usepackage[all]{xy}

\newcommand{\sn}{\mathbf{s}}

\newcommand{\Flags}{\mathrm{Fl}}
\newcommand{\Grass}{\mathrm{Gr}}
\newcommand{\Pic}{\mathrm{Pic}}
\newcommand{\C}{\mathbb{C}}

\numberwithin{equation}{section}

\newtheorem{theorem}{Theorem}[section]
\newtheorem{lemma}[theorem]{Lemma}
\newtheorem{proposition}[theorem]{Proposition}
\newtheorem{corollary}[theorem]{Corollary}

\theoremstyle{definition}
\newtheorem{example}[theorem]{Example}

\newtheorem{definition}[theorem]{Definition}
\newtheorem{remark}[theorem]{Remark}

\newcommand{\Gr}{\mathrm{Gr}}
\newcommand{\Lie}{\mathrm{Lie}}
\newcommand{\GL}{\mathrm{GL}}
\newcommand{\SL}{\mathrm{SL}}

\newcommand{\gl}{\mathfrak{gl}}
\renewcommand{\sl}{\mathfrak{sl}}
\newcommand{\nil}{\mathfrak{nil}}

\newcommand{\F}{\mathcal{F}}

\newcommand{\diag}{\qopname\relax o{diag}}

\usepackage{hyperref}

\begin{document}

\title{On homogeneous spaces for diagonal ind-groups}
\author{Lucas Fresse}
\address{Universit\'e de Lorraine, CNRS, Institut \'Elie Cartan de Lorraine, UMR 7502, Vandoeu\-vre-l\`es-Nancy, F-54506, France}
\email{lucas.fresse@univ-lorraine.fr}
\author{Ivan Penkov}
\address{Constructor University, 28759 Bremen, Germany}
\email{ipenkov@constructor.university}

\begin{abstract}
We study the homogeneous ind-spaces $\GL(\sn)/\mathbf{P}$ where $\GL(\sn)$ is a strict diagonal ind-group defined by a supernatural number $\sn$
and $\mathbf{P}$ is a parabolic ind-subgroup of $\GL(\sn)$. We construct an explicit exhaustion
of $\GL(\sn)/\mathbf{P}$ by finite-dimensional partial flag varieties.
As an application, we characterize all locally projective $\GL(\infty)$-homogeneous spaces,
and some direct products of such spaces, which are $\GL(\sn)$-homogeneous for a fixed $\sn$.
The very possibility for a $\GL(\infty)$-homogeneous space to be $\GL(\sn)$-homogeneous for a strict diagonal ind-group $\GL(\sn)$ arises from the fact that the automorphism group of a $\GL(\infty)$-homogeneous space is much larger than
$\GL(\infty)$.
\end{abstract}


\maketitle

\setcounter{tocdepth}{1}

\tableofcontents

\section{Introduction}

The ind-group $\GL(\infty)=\lim\limits_{\to}\GL(n)=\bigcup_{n\geq 1}\GL(n)$ is a most natural direct limit algebraic group, and its locally projective homogeneous spaces are quite well studied by now, see for instance \cite{Dimitrov-Penkov}, \cite{DPW}, \cite{IP}, \cite{PT}. A larger class of direct limit algebraic groups are the so called diagonal ind-groups. A rather obvious such group, non-isomorphic to $\GL(\infty)$, is the ind-group $\GL(2^\infty)=\lim\limits_\to\GL(2^n)$ where $\GL(2^{n-1})$ is embedded into $\GL(2^n)$ via the map
\[x\mapsto \begin{pmatrix} x & 0 \\ 0 & x \end{pmatrix}.\]
A general definition of a diagonal Lie algebra has been given by A. Baranov and A. Zhilinskii in \cite{BZ}, and this definition carries over in a straightforward way to classical Lie groups, producing the class of diagonal Lie groups.

Locally projective homogeneous ind-spaces of diagonal ind-groups have been studied much less extensively than those of $\GL(\infty)$, see \cite{DPW} and \cite{DP1}.
In this paper, we undertake such a study for a class of diagonal ind-groups which we call strict diagonal ind-groups of type A. These ind-groups are characterized by supernatural numbers $\sn$, and are denoted $\GL(\sn)$.
We consider reasonably general parabolic subgroups $\mathbf{P}\subset\GL(\sn)$ and describe the homogeneous ind-space $\GL(\sn)/\mathbf{P}$ as direct limits of embeddings
$$
G_{n-1}/P_{n-1}\to G_n/P_n
$$
of usual ind-varieties. Our main result is an explicit formula for the so arising embeddings, and this formula is an analogue of the formula for standard extensions introduced in \cite{PT} (and used in a particular case in \cite{Dimitrov-Penkov}).

The class of locally projective homogeneous ind-spaces of strict (and of general) diagonal ind-groups will require further detailed studies. In the current paper we restrict ourselves to the following application of the above explicit formula: we determine which locally projective homogeneous ind-spaces of $\GL(\infty)$, i.e., ind-varieties of generalized flags \cite{Dimitrov-Penkov}, are also $\GL(\sn)$-homogeneous for a given infinite supernatural number $\sn$. Furthermore, we also characterize explicitly direct products of ind-varieties of generalized flags which are $\GL(\sn)$-homogeneous.

The very possibility of an ind-variety of generalized flags being a homogeneous space for $\GL(\sn)$, where $\sn$ is an infinite supernatural number, is an interesting phenomenon, and can be seen as one possible motivation for our studies of $\GL(\sn)$-homogeneous ind-spaces. Indeed, recall the following fact for a finite-dimensional algebraic group.
If $G$ is a centerless simple algebraic group of classical type and rank at least four and $P$ is a parabolic subgroup, 
a well-known result of A. Onishchik \cite{O} implies that
the connected component of unity of the automorphism group of the homogeneous space $G/P$ coincides with $G$, except in two special cases when $G/P$ is a projective space and $G$ is a symplectic group,
and when $G/P$ is a maximal orthogonal isotropic grassmannian and $G$ is an orthogonal group of type B.
Consequently, unless $G/P$ is a projective space or a maximal isotropic grassmannian, $G/P$ cannot be a homogeneous $G'$-space for a centerless algebraic group $G'\not\cong G$.

The explanation of why the situation is very different if one replaces $G$ by the ind-group $\GL(\infty)$, is that, as shown in \cite{IP}, the automorphism group of an ind-variety of generalized flags is much larger than $\GL(\infty)$. In this way, our results provide embeddings of $\GL(\sn)$ into such automorphism groups, with the property that the action of $\GL(\sn)$ on the respective ind-variety of generalized flags is transitive.
As a corollary we obtain that a ``generic'' ind-variety of generalized flags is $\GL(\sn)$-homogeneous also for any  ind-group $\GL(\sn)$. This statement is in some sense opposite to the classical statement in the finite-dimensional case.

The paper is organized as follows.
Sections \ref{section-2}, \ref{section-3}, \ref{section-4} are devoted to preliminaries. We start by introducing the ind-groups $\GL(\sn)$ where $\sn$ is a supernatural number. We then discuss Cartan, Borel, and parabolic ind-subgroups of $\GL(\sn)$.
In Section \ref{section-3} we review the notions of linear embedding of flag varieties and standard extension of flag varieties, and in Section \ref{section-4} we recall the necessary results on ind-varieties of generalized flags. 

In Section \ref{section-5} we prove our explicit formula for embeddings of partial flag varieties $\GL(n)/Q\hookrightarrow\GL(dn)/P$ induced by pure diagonal embeddings
$\GL(n)\hookrightarrow \GL(dn)$.
In Section \ref{section-6} we use this formula to describe all $\GL(\sn)$-homogeneous ind-varieties of generalized flags.
Finally, in Section \ref{section-7} we characterize  direct products of ind-varieties of generalized flags, which are $\GL(\sn)$-homogeneous.

\subsection*{Acknowledgement}
The work of I. P. was supported in part by DFG Grant  PE 980/8-1.

\section{The ind-group $\GL(\mathbf{s})$}

\label{section-2}

\subsection{Direct systems associated to a supernatural number}

Throughout this paper we consider a fixed supernatural number $\sn$, in other words
\[\sn=\prod_{p\in\mathcal{P}}p^{\alpha_p}\]
where $\mathcal{P}$ is a (possibly infinite) set of prime numbers and $\alpha_p$ is either a positive integer or $\infty$. Moreover, we suppose that $\sn$ is infinite, hence at least one of the exponents $\alpha_p$ is infinite or the set $\mathcal{P}$ is infinite.
By $\mathcal{D}(\sn)$ we denote the set of finite divisors of $\sn$.

Let $\mathcal{A}$ be a direct system of sets with injective maps. We say that $\mathcal{A}$ is \textit{associated to the supernatural number $\sn$} if the sets in $\mathcal{A}$
\[A(s),\quad s\in\mathcal{D}(\sn)\]
are parametrized by the finite divisors of $\sn$, and the injective maps
\[\delta_{s,s'}:A(s)\hookrightarrow A(s')\]
correspond to pairs of divisors $s,s'\in\mathcal{D}(\sn)$ such that $s|s'$.
Then, if $L(\mathcal{A})=\lim\limits_\to A(s)$, the resulting map
\[\delta_s:A(s)\hookrightarrow L(A)\]
is injective for every $s\in\mathcal{D}(\sn)$.


\begin{definition}
\label{s-sequence}
We call {\em exhaustion of $\sn$} any sequence $\{s_n\}_{n\geq 1}$ of integers such that
\begin{itemize}
\item $s_n\in\mathcal{D}(\sn)$ for all $n$,
\item $s_n$ divides $s_{n+1}$ for all $n$, 
\item any $s\in\mathcal{D}(\sn)$ is a multiple of $s_n$ for some $n$.
\hfill$\blacksquare$
\end{itemize}
\end{definition}

\begin{lemma}
\label{L2.1}
Let $\{s_n\}_{n\geq 1}$ be an exhaustion of $\sn$. Then $L(\mathcal{A})$ coincides with the limit of the inductive system formed by the sets $A(s_n)$ and the maps $\delta_n=\delta_{s_n,s_{n+1}}:A(s_n)\hookrightarrow A(s_{n+1})$.
\end{lemma}

\begin{proof}
Straightforward.
\end{proof}

According to the lemma, the limit $L(\mathcal{A})$ can be described in terms of an exhaustion
\[L(\mathcal{A})=\bigcup_{n}A(s_n).\]
\begin{itemize}
\item
In the case where  $A(s)$ are vector spaces and the maps $\delta_{s,s'}$ are linear, then $L(\mathcal{A})$ is the direct limit in the category of vector spaces.
\item In the case where  $A(s)$ are algebraic varieties and the maps $\delta_{s,s'}$ are closed embeddings, the limit $L(\mathcal{A})$ is an ind-variety as defined in \cite{Safarevich} and \cite{Kumar}.
\item In the case where  $A(s)$ are algebraic groups and the maps $\delta_{s,s'}$ are group homomorphisms, the limit is both an ind-variety and a group. It is in particular an ind-group\footnote{An ind-group is an ind-variety with a group structure such that the multiplication $(x,y)\mapsto xy$ and the inversion $x\mapsto x^{-1}$ are morphisms of ind-varieties.  }.
\end{itemize}

\subsection{Definition of the groups $\GL(\sn)$ and $\SL(\sn)$}

Whenever $s,s'$ are two positive integers such that $s$ divides $s'$, we have a diagonal embedding 
$$\delta_{s,s'}:\GL(s)\to\GL(s'),\ x\mapsto \mathrm{diag}(\underbrace{x,\ldots,x}_{\mbox{\scriptsize $\frac{s'}{s}$ blocks}}).$$
We refer to the embeddings $\delta_{s,s'}$ as {\it strict diagonal embeddings}.
A more general definition of \textit{diagonal embeddings} is given, at the Lie algebra level, in \cite{BZ}. 
%

The groups $\GL(s)$ (for $s\in\mathcal{D}(\sn)$)
and the maps $\delta_{s,s'}$ (for all pairs of integers $s,s'\in\mathcal{D}(\sn)$ such that $s$ divides $s'$) form a direct system. By definition, the ind-group $\GL(\sn)$ is the limit of this direct system.

The group $\GL(\sn)$ can be viewed as the group of infinite $\mathbb{Z}_{>0}\times\mathbb{Z}_{>0}$-matrices consisting of one diagonal block of size equal to any (finite) divisor $s$ of $\sn$, repeated infinitely many times along the diagonal:
\begin{eqnarray}
\label{matrix-interpretation}
\GL(\sn) & = & \left\{
\begin{pmatrix} x & 0 & \cdots \\ 0 & x & \ddots \\ \vdots & \ddots & \ddots \\\end{pmatrix}:x\in\GL(s),\ s\in\mathcal{D}(\sn)
\right\}.
\end{eqnarray}

Similarly, we define $\SL(\sn)$ as the limit of the direct system formed by the groups $\SL(s)$ and the same maps $\delta_{s,s'}$. In fact, $\SL(\sn)$ is the derived group of $\GL(\sn)$.
By $\gl(\sn)$ and $\sl(\sn)$, we denote the Lie algebras of $\GL(\sn)$ and $\SL(\sn)$, respectively.
Thus $\sl(\sn)=[\gl(\sn),\gl(\sn)]$.

\begin{remark}
\label{R2.4}
Lemma \ref{L2.1} shows that the group $\GL(\sn)$ can be obtained through any exhaustion 
\[\GL(\sn)=\bigcup_{n}\GL(s_n)\]
where $\{s_n\}_{n\geq 1}$ is an exhaustion of $\sn$ (see Definition \ref{s-sequence}). However, the ind-group $\mathbf{GL}(\sn)$ has various other exhaustions.
If we set
\[
\mathbf{K}(n):=\underbrace{\GL(s_n)\times\cdots\times \GL(s_n)}_{\mbox{\scriptsize $\frac{s_{n+1}}{s_n}$ factors}}
\]
and
\[
\psi_n:\mathbf{K}(n)\to\mathbf{K}(n+1),\ (x_1,\ldots,x_{d_n})\mapsto (\underbrace{\diag(x_1,\ldots,x_{d_n}),\ldots,\diag(x_1,\ldots,x_{d_n})}_{\frac{s_{n+2}}{s_{n+1}}\mbox{\scriptsize \ terms} }),
\]
then the direct system $\{\mathbf{K}(n)\stackrel{\psi_n}{\to}\mathbf{K}(n+1)\}$ intertwines in a natural way with the direct system $\{\GL(s_n)\stackrel{\delta_{s_n,s_{n+1}}}{\longrightarrow}\GL(s_{n+1})\}$
considered above.
This yields an equality
\[\mathrm{GL}(\sn)=\lim_\to \GL(s_n)=\lim_\to\mathbf{K}(n).\]
\hfill$\blacksquare$

\end{remark}

We say that two exhaustions $\mathbf{G}=\bigcup_nG_n=\bigcup_n G'_n$ of a given ind-group are \textit{equivalent} if there are $n_0\geq 1$ and a commutative diagram
$$
\xymatrix{
G_{n_0} \ar[d]^\sim \ar[r] & \cdots \ar[r] & G_n \ar[d]^\sim\ar[r] & G_{n+1} \ar[d]^\sim\ar[r] & \cdots \\
G'_{n_0} \ar[r] & \cdots \ar[r] & G'_n \ar[r] & G'_{n+1} \ar[r] & \cdots \\
}
$$
such that the vertical arrows are isomorphisms of algebraic groups and the horizontal arrows are the embeddings of the exhaustions.

\begin{lemma}
\label{L2.5-new}
{\rm (a)} Any exhaustion of $\SL(\sn)$ by almost simple, simply connected algebraic groups is equivalent to $\{\SL(s_n),\delta_{s_n,s_{n+1}}\}_{n\geq 1}$ for an exhaustion $\{s_n\}_{n\geq 1}$ of $\sn$.

{\rm (b)} Any exhaustion of $\GL(\sn)$ by classical groups (i.e. by groups of the form $\GL(n)$, $\SL(n)$, $\mathrm{SO}(n)$, or $\mathrm{Sp}(n)$) is equivalent to $\{\GL(s_n),\delta_{s_n,s_{n+1}}\}_{n\geq 1}$ for an exhaustion $\{s_n\}_{n\geq 1}$ of $\sn$.

%
\end{lemma}

\begin{proof}
(a) It suffices to prove the claim at the level of Lie algebras.
Let $\sl(\sn)=\bigcup_n \mathfrak{g}_n$ be an exhaustion by simple Lie algebras, hence of classical type for $n$ large enough. There is a subsequence $\{\mathfrak{g}_{k_n}\}_{n\geq 1}$ and an exhaustion $\{s_n\}_{n\geq 1}$ of $\sn$ such that we have a commutative diagram of embeddings
$$
\xymatrix{\sl(s_n) \ar[r]^{\delta_{s_n,s_{n+1}}} \ar[d]^{\eta_n} & \sl(s_{n+1}) \ar[d]^{\eta_{n+1}} \ar[r] & \cdots \\ \mathfrak{g}_{k_n} \ar[ru]^{\xi_n} \ar[r] & \mathfrak{g}_{k_{n+1}} \ar[ru] \ar[r] & \cdots }
$$
By \cite[Lemma 2.7]{BZ}, the embeddings $\eta_n$ and $\xi_n$ are diagonal, in the sense that there is an isomorphism of  
$\sl(s_n)$-modules
$$W_n\cong V_n^{\oplus t}\oplus {V_n^*}^{\oplus r}\oplus \C^{\oplus s} $$
and an isomorphism of $\mathfrak{g}_{k_n}$-modules
$$
V_{n+1}\cong W_n^{\oplus t'}\oplus {W_n^*}^{\oplus r'}\oplus \C^{\oplus s'} $$
for some triples of nonnegative integers $(t,r,s)$ and $(t',r',s')$,
where
$V_n$ and $W_n$ denote the natural representations of $\sl(s_n)$ and $\mathfrak{g}_{k_n}$,
and $\C$ is a trivial representation.
Also since $\delta_{s_n,s_{n+1}}$ is strict diagonal, we have an isomorphism of  
$\sl(s_n)$-modules
\begin{equation}
\label{Vn+1Vn}
V_{n+1}\cong\underbrace{V_n\oplus\ldots\oplus V_n}_{\mbox{\scriptsize $\frac{s_{n+1}}{s_n}$ copies}}.
\end{equation}

Arguing by contradiction, assume that $\mathfrak{g}_{k_n}$ is not of type A. Then \cite[Proposition 2.3]{BZ} implies that $t=r$. Moreover, $t'+r'>0$ since otherwise $V_{n+1}$ would be a trivial representation of $\sl(s_n)$. Altogether this implies that $V_n^*$ is isomorphic to a direct summand of $V_{n+1}$ considered as an $\sl(s_n)$-module, which is impossible in view of (\ref{Vn+1Vn}). We conclude that $\mathfrak{g}_{k_n}$ is of type A for all $n$.

Moreover, from (\ref{Vn+1Vn}),
we obtain $s=s'=1$ and either $r=r'=0$ or $t=t'=0$. Up to replacing $\mathfrak{g}_{k_n}=\sl(W_n)$ by $\sl(W_n^*)$,
we can assume that $r=r'=0$, and so $\mathfrak{g}_{k_n}\cong \sl(s'_{k_n})$ for some integer such that $s_n|s'_{k_n}$, $s'_{k_n} |s_{n+1}$, and the embedding $\mathfrak{g}_{k_n}\hookrightarrow\mathfrak{g}_{k_{n+1}}$ is induced by $\delta_{s'_{k_n},s'_{k_{n+1}}}$.

If $k:=k_n+1<k_{n+1}$, we get a commutative diagram
$$
\xymatrix{
\sl(s'_{k_n}) \ar[d]^\sim \ar[rr]^{\delta_{s'_{k_n},s'_{k_{n+1}}}} & & \sl(s'_{k_{n+1}}) \ar[d]^\sim  \\
\mathfrak{g}_{k_n} \ar[r] & \mathfrak{g}_k \ar[r] &  \mathfrak{g}_{k_{n+1}}}
$$
where the horizontal arrows are embeddings. Relying as above on  \cite[Proposition 2.3]{BZ},
we get that $\mathfrak{g}_k$ is necessarily of type A, and up to replacing $\mathfrak{g}_k=\sl(W)$ by $\sl(W^*)$, we can assume that $\mathfrak{g}_k\cong \sl(s'_k)$ for some $s'_k$ with $s'_{k_n}|s'_k$, $s'_k|s'_{k_{n+1}}$, and that the embeddings $\mathfrak{g}_{k_n}\hookrightarrow \mathfrak{g}_k\hookrightarrow \mathfrak{g}_{k_{n+1}}$ are induced by  $\delta_{s'_{k_n},s'_k}$ and 
$\delta_{s'_{k},s'_{k_{n+1}}}$.

By iterating the reasoning, we obtain an exhaustion $\{s'_n\}_{n\geq 1}$ of $\sn$ such that the exhaustions
$\sl(\sn)=\bigcup_n\mathfrak{g}_n$ and $\sl(\sn)=\bigcup_n\sl(s'_n)$ are equivalent. This shows (a).

(b) From (a) it follows that for every $n$, the derived group $(G_n,G_n)$ is isomorphic to $\SL(s_n)$ and, after identifying $(G_n,G_n)$ with $\SL(s_n)$ and $(G_{n+1},G_{n+1})$ with $\SL(s_{n+1})$, the map $(G_n,G_n)\hookrightarrow (G_{n+1},G_{n+1})$ becomes the restriction of $\delta_{s_n,s_{n+1}}$. This implies that $G_n$ is either isomorphic to $\SL(s_n)$ or to $\GL(s_n)$. For $n\geq 1$ large enough, $G_n$ has to contain the center $Z(\GL(\sn))$, which is isomorphic to $\mathbb{C}^*$. Since the connected component of the center of $\SL(s_n)$ is trivial, this forces $G_n\cong\GL(s_n)$. Moreover, since $G_n=Z(G_n)(G_n,G_n)$ and the embedding $G_n\hookrightarrow G_{n+1}$ maps $Z(G_n)=Z(\GL(\sn))$ into $Z(G_{n+1})$, we deduce that this embedding $G_n\hookrightarrow G_{n+1}$ coincides with $\delta_{s_n,s_{n+1}}:\GL(s_n)\hookrightarrow \GL(s_{n+1})$
after suitably identifying $G_n$ with $\GL(s_n)$ and $G_{n+1}$ with $\GL(s_{n+1})$..
\end{proof}

The following statement is a corollary of the classification of general diagonal Lie algebras \cite{BZ}. We give a proof for the sake of completeness.

\begin{proposition}
{\rm (a)} If $\sn$ and $\sn'$ are two different infinite supernatural numbers, then the ind-groups $\GL(\sn)$ and $\GL(\sn')$ (resp. $\SL(\sn)$ and $\SL(\sn')$) are not isomorphic.

{\rm (b)} If $\sn$ is an infinite supernatural number, then $\GL(\sn)$ is not isomorphic to $\GL(\infty)$, and $\SL(\sn)$ is not isomorphic to $\SL(\infty)$.
\end{proposition}

\begin{proof}
(a) Since $\SL(\cdot)$ is the derived group of $\GL(\cdot)$, it suffices to establish the claim concerning $\SL(\sn)$ and $\SL(\sn')$. 
Assume there is an isomorphism of ind-groups $\varphi:\SL(\sn')\to \SL(\sn)$. Then any exhaustion $\{s'_n\}$ of $\sn'$ yields an exhaustion $\SL(\sn)=\bigcup_n \varphi(\SL(s'_n))$ of the group $\SL(\sn)$, and Lemma \ref{L2.5-new} implies $\sn=\sn'$, a contradiction.

(b) By definition, $\SL(\infty)$ has an exhaustion by the groups $\SL(n)$ ($n\geq 1$) via the standard embeddings $\SL(n)\to\SL(n+1)$, $x\mapsto\begin{pmatrix}x&0\\0&1\end{pmatrix}$. Clearly this exhaustion is not equivalent to $\{\SL(s_n),\delta_{s_n,s_{n+1}}\}_{n\geq 1}$ for any exhaustion $\{s_n\}_{n\geq 1}$ of $\sn$. Therefore, the ind-groups $\SL(\sn)$ and $\SL(\infty)$ are not isomorphic by Lemma \ref{L2.5-new}\,(a). The same argument shows that $\GL(\sn)$ and $\GL(\infty)$ are not isomorphic.
\end{proof}

\subsection{Parabolic and Borel subgroups}
\label{section-2.3}
An ind-subgroup $\mathbf{H}\subset\GL(\sn)$ is said to be a {\it (locally splitting) Cartan subgroup} if
there is an exhaustion $\GL(\sn)=\bigcup_n G_n$ by classical groups
such that $G_n\cap\mathbf{H}$ is a Cartan subgroup of $G_n$ for all $n$.
For instance, the subgroup of invertible periodic diagonal matrices 
in the realization (\ref{matrix-interpretation})  is a Cartan subgroup of $\GL(\sn)$.


If $\mathbf{P}$ is an ind-subgroup of $\GL(\sn)$, 
 then the quotient $\GL(\sn)/\mathbf{P}$ is an ind-variety obtained as the direct limit of the quotients $\GL(s)/\mathbf{P}(s)$ for $s\in\mathcal{D}(\sn)$. 

For the purposes of this paper, we say that an ind-subgroup $\mathbf{P}\subset\GL(\sn)$ 
is a \textit{parabolic subgroup}
if there exists an exhaustion $\GL(\sn)=\bigcup_n G_n$ by classical groups such that
$G_n\cap \mathbf{P}$ is a parabolic subgroup of $G_n$ for all $n$ (cf. \cite{DP1}).
This implies in particular that the ind-variety $\GL(\sn)/\mathbf{P}$ is \textit{locally projective} as it has an exhaustion
\[\GL(\sn)/\mathbf{P}=\bigcup_n G_n/(G_n\cap \mathbf{P})\]
by projective varieties.
If, in addition,
the unipotent radical of $G_n\cap \mathbf{P}$ is contained in the unipotent radical of $G_{n+1}\cap \mathbf{P}$ for every $n$, then we say that $\mathbf{P}$ is a {\em strong parabolic subgroup}.

An ind-subgroup $\mathbf{B}\subset\GL(\sn)$ 
is said to be a \textit{Borel subgroup} if it is locally solvable 
and parabolic. This means equivalently that there is an exhaustion $\GL(\sn)=\bigcup_n G_n$ as above for which $G_n\cap \mathbf{B}$ is a Borel subgroup of $G_n$ for all $n$. Note that a Borel subgroup is necessarily a strong parabolic subgroup.

\begin{lemma}
A subgroup $\mathbf{G}'$ of $\GL(\sn)$ is a Cartan (respectively, parabolic or Borel) subgroup of $\mathbf{G}$
if and only if there is an exhaustion $\{s_n\}_{n\geq 1}$ of $\sn$ such that 
for every $n$ the intersection $\mathbf{G}'\cap\GL(s_n)$ is a Cartan (respectively, parabolic or Borel) subgroup of $\GL(s_n)$.
\end{lemma}

\begin{proof}
This follows from Lemma \ref{L2.5-new}.
\end{proof}



The following example shows that for a given parabolic subgroup $\mathbf{P}\subset\GL(\sn)$, the property that the group $G_n\cap\mathbf{P}$ is a parabolic subgroup of $G_n$ 
may no longer hold for a refinement of the exhaustion used to define $\mathbf{P}$.

\begin{example}
\label{E2.5}
Let $\sn=2^\infty$, $s_n=2^{2n-2}$, and $s'_n=2^{n-1}$. 
Then both $\{s_n\}_{n\geq 1}$ and $\{s'_n\}_{n\geq 1}$ are exhaustions of $\sn$, and $\{s'_n\}_{n\geq 1}$ is a refinement of $\{s_n\}_{n\geq 1}$. Let $H_n\subset \GL(s_n)$ be the subgroup of diagonal matrices.
We define a Borel subgroup $B_n\subset\GL(s_n)$ that contains $H_n$, by induction in the following way:
$B_1:=\GL(1)$, and
\[B_{n+1}:=\left(\begin{array}{cccc}
B_n & * & * & * \\ 0 & B_n & * & * \\ 0 & 0 & B_n & 0 \\ 0 & 0 & * & B_n \end{array}\right)\]
 for  $n\geq 2$, where all the blocks are square matrices of size $s_n$.
Then $B_{n+1}\cap\GL(s_n)=B_n$ for all $n$, which implies that
$\mathbf{B}=\bigcup_{n\geq 1} B_n$
is a well-defined Borel subgroup of $\GL(\sn)$ arising from the exhaustion $\{s_n\}_{n\geq 1}$ of $\sn$. However, for all $n$,
\[\mathbf{B}\cap\GL(s'_{2n})=\begin{pmatrix} B_n & 0 \\ 0 & B_n\end{pmatrix}\]
is not a Borel subgroup (nor a parabolic subgroup) of $\GL(s'_{2n})$.
\hfill $\blacksquare$

\end{example}





\section{On embeddings of flag varieties}
\label{section-3}

In this section we review some preliminaries
on finite-dimensional (partial) flag varieties.
In particular, we recall the notions of linear embedding and standard extension  introduced in \cite{PT}.

\subsection{Grassmannians and (partial) flag varieties}

\label{section-3.1}

Let $V$ be a finite-dimensional vector space. For an integer $0\leq p\leq \dim V$, we denote by $\Grass(p;V)$ the grassmannian of $p$-dimensional subspaces in $V$. This grassmannian can be realized as a projective variety by the Pl\"ucker embedding $\Grass(p;V)\hookrightarrow \mathbb{P}(\bigwedge^p V)$. Moreover, the Picard group $\Pic(\Grass(p;V))$ is isomorphic to $\mathbb{Z}$ with generator $\mathcal{O}_{\Grass(p;V)}(1)$,
the pull-back of the line bundle $\mathcal{O}(1)$ on $\mathbb{P}(\bigwedge^p V)$. 

For a sequence of integers $0<p_1<\ldots<p_{k-1}<p_k<\dim V$, we denote by $\Flags(p_1,\ldots,p_k;V)$ the variety of (partial) flags
\[
\Flags(p_1,\ldots,p_k;V)=\{(V_1,\ldots,V_k)\in\Grass(p_1;V)\times\cdots\times\Grass(p_k;V):V_1\subset\ldots\subset V_k\}.
\]
We have 
\[\Pic(\Flags(p_1,\ldots,p_k;V))\cong \mathbb{Z}^{k}.\]
If we let $L_i$ be the pull-back 
\[L_i=\mathrm{proj}_i^*\mathcal{O}_{\Grass(p_i;V)}(1)\]
along the projection 
\[\mathrm{proj}_i:\Flags(p_1,\ldots,p_k;V)\to \Grass(p_i;V)\]
(for $i=1,\ldots,k$),
then $[L_1],\ldots,[L_{k}]$ is a set of generators of the Picard group.

By \textit{embedding of flag varieties} we mean a closed immersion
\[\varphi:X=\Flags(p_1,\ldots,p_k;V)\hookrightarrow Y=\Flags(q_1,\ldots,q_\ell;W).\]
If $\F=\{F_1,\ldots,F_k\}\in X$ is a variable point, we set
$$C_i(\varphi)=\bigcap_{\mathcal{F}\in X}\varphi(\F)_i.$$
Then $C_1(\varphi)\subset\ldots\subset C_\ell(\varphi)$ is a chain of subspaces of $W$ with possible repetitions.
We define the {\em support} of $\varphi$ to be the set of indices $i\in\{1,\ldots,\ell\}$ such that
$\dim C_i(\varphi)<q_i$.

\subsection{Linear embedding}

\label{section-3.3}

Let 
\[Q:=\Grass(q_1;W_1)\times\cdots\times\Grass(q_{\ell};W_{\ell})\]
where
$W_1,\ldots,W_{\ell}$ is a sequence of vector spaces and $0<q_j<\dim W_j$ for all $j$. Consider an embedding
\[\psi:X=\Flags(p_1,\ldots,p_k;V)\to Q.\]
We use the notation of the previous section for $X$.
The Picard group of $Q$ is isomorphic to $\mathbb{Z}^{\ell}$, with generators associated to the line bundles $M_j=\mathrm{proj}_j^*\mathcal{O}_{\Grass(q_j;W_j)}(1)$.

\begin{definition}
We say that the embedding $\psi$ is {\em linear} if 
we have
\[[\psi^*M_j]=0\quad\mbox{or}\quad[\psi^*M_j]\in\{[L_1],\ldots,[L_{k}]\}\]
for all $j\in\{1,\ldots,\ell\}$.
\hfill $\blacksquare$
\end{definition}


Let 
\[\varphi:X=\Flags(p_1,\ldots,p_k;V)\hookrightarrow Y=\Flags(q_1,\ldots,q_\ell;W)\]
be an embedding of flag varieties.
The following definition is equivalent to \cite[Definition 2.1]{PT}.

\begin{definition}
The embedding $\varphi$ is said to be {\em linear} if the composed embedding
$\psi=\pi\circ\varphi$ is linear, where $\pi:=\prod_{j=1}^\ell\mathrm{proj}_j:Y\to \prod_{j=1}^{\ell}\Grass(q_j;W)$.
\hfill $\blacksquare$
\end{definition}

\subsection{Standard extension}

\begin{definition}[\cite{PT}] \label{def-SE} 
(a)
The embedding $\varphi:\Flags(p_1,\ldots,p_k;V)\hookrightarrow\Flags(q_1,\ldots,q_\ell;W)$ is said to be a {\em strict standard extension} if there are 
\begin{itemize}
\item a decomposition of vector spaces $W=V'\oplus Z$ with a linear isomorphism $\varepsilon:V\stackrel{\sim}{\to}V'$,
\item a chain of subspaces
$Z_1\subset \ldots\subset Z_\ell$ of $Z$ (with possible repetitions),
\item a nondecreasing map $\kappa:\{1,\ldots,\ell\}\to\{0,1,\ldots,k,k+1\}$,
\end{itemize}
such that
\begin{equation}
\label{formula:standard-extension}
\varphi\big(\{V_1,\ldots, V_k\}\big)
=\{\varepsilon( V_{\kappa(1)})+Z_1,\ldots, \varepsilon(V_{\kappa(\ell-1)})+Z_{\ell-1}
, \varepsilon(V_{\kappa(\ell)})+Z_\ell\}
\end{equation}
where $V_0:=0$ and $V_{k+1}:=V$.

(b) More generally, we say that $\varphi$ is a \textit{standard extension} if $\varphi$ itself is a strict standard extension or its composition of $\varphi$ with the duality map $\Flags(q_1,\ldots,q_\ell;W)\to \Flags(\dim W-q_1,\ldots,\dim W-q_\ell;W^*)$
is a strict standard extension.
\hfill $\blacksquare$
\end{definition}

\begin{remark}
Since the map $\varphi$ of (\ref{formula:standard-extension}) is an embedding of flag varieties, the following conditions must hold: $1,\ldots,k$ have preimages by $\kappa$, and the map $j\in\{1,\ldots,\ell\}\mapsto(\kappa(j),Z_j)$ is injective and does not contain $(0,0)$ nor $(k+1,Z)$ in its image.
\hfill $\blacksquare$
\end{remark}

Note that, if $\varphi$ is a strict standard extension, then $C_i(\varphi)=Z_i$ for all $i\in\{1,\ldots,\ell\}$,
and the support of $\varphi$ is the interval\ $\kappa^{-1}([1,k])$. 

Also, a composition of standard extensions is a standard extension.

\begin{example}
\label{E3.5}
Let $W=V\oplus Z$, where $\dim Z=d$. For $1\leq k_0\leq k+1$, we consider the embeddings
\begin{eqnarray*}
\varphi:\mathrm{Fl}(p_1,\ldots,p_k;V) & \hookrightarrow &  \Flags(q_1,\ldots,q_k;W) \\
\{V_1,\ldots,V_k\} & \mapsto & \{V_1,\ldots,V_{k_0-1}, V_{k_0}+Z,\ldots,V_k+Z\}
\end{eqnarray*}
and
\begin{eqnarray*}
\bar\varphi:\mathrm{Fl}(p_1,\ldots,p_k;V) & \hookrightarrow &  \Flags(\bar{q}_1,\ldots,\bar{q}_{k+1};W) \\
\{V_1,\ldots,V_k\} & \mapsto & \{V_1,\ldots,V_{k_0-1}, V_{k_0-1}+Z,V_{k_0}+Z,\ldots,V_k+Z\}
\end{eqnarray*}
where 
$$q_i=\left\{\begin{array}{ll}
p_i & \mbox{if $1\leq i<k_0$,} \\ p_i+d & \mbox{if $k_0\leq i\leq k$}
\end{array}\right.\quad\mbox{and}\quad \bar{q}_i=\left\{
\begin{array}{ll}
p_i & \mbox{if $1\leq i<k_0$,} \\ p_{i-1}+d & \mbox{if $k_0\leq i\leq k+1$.}
\end{array}
\right.$$
Here, we still use the convention that $V_0:=0$ and $V_{k+1}:=V$,
and we set accordingly $p_0:=0$ and $p_{k+1}:=\dim V$.
Then $\varphi$ and $\bar\varphi$ are strict standard extensions, associated with the respective chains of subspaces
$$
\underbrace{0\subset\ldots\subset 0}_{\mbox{\scriptsize $k_0-1$ times}}\subset \underbrace{Z\subset\ldots\subset Z}_{\mbox{\scriptsize $k+1-k_0$ times}}
\quad\mbox{and}\quad
\underbrace{0\subset\ldots\subset 0}_{\mbox{\scriptsize $k_0-1$ times}}\subset \underbrace{Z\subset\ldots\subset Z}_{\mbox{\scriptsize $k+2-k_0$ times}}
$$
and respective maps $\kappa$ and $\bar{\kappa}$, where $\kappa(i)=i$ for all $i$, $\bar\kappa(i)=i$ for $i\leq k_0-1$, $\bar\kappa(i)=i-1$ for $i\geq k_0$.\hfill $\blacksquare$
\end{example}

\begin{remark}
Every strict standard extension is the composition of, possibly several, maps $\varphi$ and $\bar\varphi$ as in Example \ref{E3.5}.
\hfill$\blacksquare$
\end{remark}

\section{A review of generalized flags}
\label{section-4}

\subsection{Generalized flags}

\label{section-4.1}

Let $V$ be an infinite-dimensional vector space of countable dimension and let $E=\{e_1,e_2,\ldots\}$ be a basis of $V$. By $\langle S\rangle$, we denote the span of vectors in a subset $S\subset V$. Following \cite{Dimitrov-Penkov}, we call {\em generalized flag} a collection $\mathcal{F}$ of subspaces of $V$ that satisfies the following conditions:
\begin{itemize}
\item $\mathcal{F}$ is totally ordered by inclusion;
\item every subspace $F\in \mathcal{F}$ has an immediate predecessor or an immediate successor in $\mathcal{F}$;
\item $V\setminus\{0\}=\bigcup_{(F',F'')}(F''\setminus F')$, where the union is over pairs of consecutive subspaces in~$\mathcal{F}$.
\end{itemize}
Moreover, a generalized flag $\mathcal{F}$ is said to be \textit{$E$-compatible} if every subspace $F\in\mathcal{F}$ is spanned by elements of $E$. An $E$-compatible generalized flag $\F$ can be encoded by 
a (not order preserving) surjective map $\sigma:\mathbb{Z}_{>0}\to A$ onto a totally ordered set $(A,\leq)$ such that $\mathcal{F}=\{F'_a,F''_a\}_{a\in A}$ where $F'_a=\langle e_k:\sigma(k)<a\rangle$ and $F''_a=\langle e_k:\sigma(j)\leq a\rangle$.
More generally, a generalized flag $\F$ is said to be
\textit{weakly $E$-compatible} if it is $E'$-compatible for some basis  $E'$ of $V$ differing from $E$ in finitely many vectors.



Let $$\GL(E)=\{g\in \mathrm{GL}(V):g(e_k)=e_k\ \mbox{for all but finitely many $k$}\}.$$ 
Then $\GL(E)$ is an ind-group, isomorphic to the finitary classical ind-group $\GL(\infty)$.
The group $\GL(E)$ acts on the set of all weakly $E$-compatible generalized flags. Furthermore, it is established in \cite{Dimitrov-Penkov} that weakly $E$-compatible generalized flags $\mathcal{F}$ of $V$ are in one-to-one correspondence with splitting parabolic subgroups $\mathbf{P}\subset\GL(E)$. More precisely, the map
\[\mathcal{F}\mapsto \mathbf{P}=\mathrm{Stab}_{\GL(E)}(\mathcal{F})\]
is a bijection between these two sets.

By a \textit{natural representation} of $\GL(\sn)$ we mean a direct limit of natural representations of $\GL(s)$ for $s\in\mathcal{D}(\sn)$. Two natural representations do not have to be isomorphic; see \cite{HS}.

Assume now that $V$ is a natural representation for $\GL(\sn)$, $\mathbf{H}\subset\GL(\sn)$ is a Cartan subgroup such that there is a basis $E$ of $V$ consisting of eigenvectors of $\mathbf{H}$.
The group $\GL(\sn)$ acts in a natural way on the generalized flags in $V$, and a generalized flag is $E$-compatible if and only if it is $\mathbf{H}$-stable.
However,  generalized flags are less suited for describing parabolic subgroups of $\GL(\sn)$ than for describing parabolic subgroups of $\GL(\infty)\cong\GL(E)$, since the stabilizer of a generalized flag in $\GL(\sn)$ is not always a parabolic subgroup. Moreover, there are parabolic subgroups of $\GL(\sn)$ which cannot be realized as stabilizers of generalized flags in a prescribed natural representation. These observations are illustrated by the following two examples.

\begin{example}
For every $n\geq 0$, we define inductively a subset $I_n\subset\{1,\ldots,2^{n+1}\}$ by setting
\[I_0:=\{1\}\subset\{1,2\},\qquad I_n:=I_{n-1}\cup\{2^n+i:i\in\{1,\ldots,2^n\}\setminus I_{n-1}\}\ \ \mbox{for $n\geq 1$}.\]
Note that $\{I_n\}_{n\geq 0}$ is a nested sequence of sets, and let $I:=\bigcup_{n\geq 0}I_n$. 
For $V=\langle e_1,e_2,\ldots\rangle$ as above, put
\[W:=\langle e_i:i\in I\rangle.\]
Thus $\mathcal{F}:=\{0\subset W\subset V\}$ is a generalized flag. 

By Lemma \ref{L2.5-new}\,(b), any exhaustion of $\GL(2^\infty)$ by classical groups is equivalent to $\{\GL(s_n),\delta_{s_n,s_{n+1}}\}_{n\geq 1}$ for an exhaustion $\{s_n=2^{k_n}\}_{n\geq 1}$ of $\sn$. Every element
$g\in\GL(2^{k_n})$  stabilizing $\mathcal{F}$ should be such that the blockwise diagonal matrix
\[\begin{pmatrix} g & 0 \\ 0 & g\end{pmatrix}\]
stabilizes $\langle e_i:i\in I_{k_n}\rangle=\langle e_i:i\in I_{k_n-1}\rangle\oplus\langle e_{2^{k_n-1}+i}:i\in\{1,\ldots,2^{k_n}\}\setminus I_{k_n-1}\rangle$, hence $g$ should stabilize both subspaces $\langle e_i:i\in I_{k_n-1}\rangle$ and $\langle e_i:i\in\{1,\ldots,2^{k_n}\}\setminus I_{k_n-1}\rangle$. This implies that the stabilizer of $W$ in $\GL(2^{k_n})$ is not a parabolic subgroup, for all $n\geq 1$. Therefore, $\mathrm{Stab}_{\GL(2^\infty)}(\mathcal{F})$ is not a parabolic subgroup of $\GL(2^\infty)$.
\hfill $\blacksquare$
\end{example}

\begin{example}
(a) Let $V=\bigcup_n\mathbb{C}^{2^n}$ be seen as a natural representation of $\GL(2^\infty)$ where the embedding $\mathbb{C}^{2^n}\cong\mathbb{C}^{2^n}\times\{0\}^{2^n}\subset \mathbb{C}^{2^{n+1}}$ is considered.
For $n\geq 1$, let $P_n\subset\GL(2^n)$ be the stabilizer of $L_n:=\{0\}^{2^{n}-1}\times\mathbb{C}$, the line spanned by the $2^{n}$-th vector of the standard basis of $\mathbb{C}^{2^n}$.
Then $P_{n+1}\cap\GL(2^n)=P_n$ for all $n\geq 1$, hence $\mathbf{P}:=\bigcup_{n\geq 1}P_n$ is a parabolic subgroup of $\GL(2^\infty)$. However, $\mathbf{P}$ acts transitively on the nonzero vectors of $V$, so that there is no nonzero proper subspace of $V$ which is stable by $\mathbf{P}$. Therefore, $\mathbf{P}$ cannot be realized as the stabilizer of a generalized flag in $V$.

(b) If in part (a) we replace the embeddings defining the structure of natural representation on $V$ by $\mathbb{C}^{2^n}\cong\{0\}^{2^n}\times\mathbb{C}^{2^n}\subset\mathbb{C}^{2^{n+1}}$, then $L_n=L_1$ for all $n\geq 1$ and the parabolic subgroup $\mathbf{P}$ of (a) becomes the stabilizer of the generalized flag $\{0\subset L_1\subset V\}$.
\hfill $\blacksquare$
\end{example}


\subsection{Ind-varieties of generalized flags} 

\label{section-4.2}

\begin{definition}
(a) Two generalized flags $\mathcal{F}$ and $\mathcal{G}$ are said to be \textit{$E$-commensurable} \cite{Dimitrov-Penkov} if $\mathcal{F}$ and $\mathcal{G}$ are weakly $E$-compatible and
there is an isomorphism of totally ordered sets $\phi:\mathcal{F}\to\mathcal{G}$ and there is a finite-dimensional subspace $U\subset V$ such that, for all $F\in\mathcal{F}$, $F+U=\phi(F)+U$ and $\dim F\cap U=\dim \phi(F)\cap U$.

(b) Given an $E$-compatible generalized flag $\mathcal{F}$, we define $\Flags(\mathcal{F},E)$ as the set of all generalized flags which are $E$-commensurable with $\mathcal{F}$.
\hfill $\blacksquare$
\end{definition}

Let $\mathcal{F}$ be an $E$-compatible generalized flag.
We now recall the ind-variety structure on
$\mathrm{Fl}(\mathcal{F},E)$ \cite{Dimitrov-Penkov}. 
To do this, we write $E=\{e_k\}_{k\geq 1}$ and, for $n\geq 1$, set $V_n:=\langle e_1,\ldots,e_n\rangle$. 
The collection of subspaces $\{F\cap V_n:F\in\mathcal{F}\}$ determines a  flag $F^{(n)}_1\subset\ldots\subset F^{(n)}_{p_n-1}$ in $F_{p_n}^{(n)}:=V_n$;
furthermore we set $d_i^{(n)}:=\dim F^{(n)}_i$ and 
$$X_n:=\Flags(d_1^{(n)},\ldots,d_{p_n-1}^{(n)};V_n).$$
We define an embedding $\eta_n:X_n\to X_{n+1}$ in the following way.
Let $i_0\in\{1,\ldots,p_{n+1}\}$ be minimal such that $e_{n+1}\in F^{(n+1)}_{i_0}$.
We have either $p_{n+1}=p_n$ or $p_{n+1}=p_n+1$. In the former case we set
$$
\eta_n:\{M_1,\ldots,M_{p_n-1}\}\mapsto \{M_1,\ldots,M_{i_0-1},M_{i_0}\oplus\langle e_{n+1}\rangle,\ldots,M_{p_n-1}\oplus\langle e_{n+1}\rangle\}.
$$
In the latter case, we define
$$
\eta_n:\{M_1,\ldots,M_{p_n-1}\}\mapsto \{M_1,\ldots,M_{i_0-1},M_{i_0-1}\oplus\langle e_{n+1}\rangle,\ldots,M_{p_n-1}\oplus\langle e_{n+1}\rangle\}.
$$

\begin{proposition}[\cite{Dimitrov-Penkov}]
{\rm (a)} The maps $\{\eta_n\}_{n\geq 1}$ are strict standard extensions and they yield an exhaustion $\Flags(\mathcal{F},E)=\bigcup_{n\geq 1}X_n$. This endows $\Flags(\mathcal{F},E)$ with a structure of locally projective ind-variety.

{\rm (b)} If $\mathbf{P}=\mathrm{Stab}_{\GL(E)}(\mathcal{F})$, then there is a natural isomorphism of ind-varieties $\GL(E)/\mathbf{P}\stackrel{\sim}{\to} \Flags(\mathcal{F},E)$.
\end{proposition}

Note also that, up to isomorphism, the ind-variety $\Flags(\mathcal{F},E)$ only depends on the {\em type} of $\mathcal{F}$, i.e., on  the isomorphism type of the totally ordered set $(\mathcal{F},\subset)$ and on the dimensions $\dim F''/F'$ of the quotients of consecutive subspaces in $\mathcal{F}$.






%
%
%


\section{Embedding of flag varieties arising from diagonal embedding of groups}
\label{section-5}

In this section we study embeddings of flag varieties induced by strictly diagonal embeddings of general linear groups. 

Let us fix the following data:\begin{itemize}
\item positive integers $m<n$ such that $m$ divides $n$, and $d:=\frac{n}{m}$;
\item $\GL(m)$ seen as a subgroup of $\GL(n)$ through the diagonal embedding $$x\mapsto\mathrm{diag}(x,\ldots,x);$$
\item a decomposition of the natural representation $V:=\C^n$ of $\GL(n)$ as 
$$V=W^{(1)}\oplus\ldots\oplus W^{(d)}$$ where $W^{(i)}:=\{0\}^{(i-1)m}\times\C^m\times\{0\}^{(d-i)m}$; let  $\chi_i:W:=\C^m\to W^{(i)}$ be the natural isomorphism.
For a 
subspace $M\subset W$, we write 
$M^{(i)}:=\chi_i(M)$.
\end{itemize}

\subsection{Restriction of parabolic subgroup}
\label{section-5.1}
Let $\{e_1,\ldots,e_n\}$ be a basis of $V$ such that $\{e_1,\ldots,e_m\}$ is a basis of $W^{(1)}\cong W$.
By $H=H(n)\subset\GL(n)$ we denote the  maximal torus for which $e_1,\ldots,e_n$ are eigenvectors. Then $H':=H\cap\GL(m)$ is a maximal torus of $\GL(m)$.

A parabolic subgroup $P=P(n)\subset \GL(n)$ that contains $H$ is the stabilizer of a flag
\[
\mathcal{F}_\alpha=\{\langle e_i:\alpha(i)\leq j\rangle\}_{j=1}^{p-1}
\]
for some surjective map
$\alpha:\{1,\ldots,n\}\to \{1,\ldots,p\}$.
The following statement determines under what condition the intersection $P\cap \GL(m)$ is a parabolic subgroup.

\begin{lemma}
\label{L5.1}
Consider the map 
\[\beta:\{1,\ldots,m\}\to\{1,\ldots,p\}^d,\ r\mapsto (\alpha(r),\alpha(m+r),\ldots,\alpha((d-1)m+r)),\]
and denote by $\mathcal{I}$ the image of $\beta$.
Let $\leq$ denote the partial order on $\{1,\ldots,p\}^d$ such that $(x_1,\ldots,x_d)\leq(y_1,\ldots,y_d)$ if $x_i\leq y_i$ for all $i$.

{\rm (a)} The intersection $Q:=P\cap\GL(m)$ is a parabolic subgroup of $\GL(m)$ if and only if $\leq$ restricts to a total order on $\mathcal{I}$. 
Moreover, letting $b_1,\ldots,b_q$ be the elements of $\mathcal{I}$ written in increasing order, we have
\[Q=\mathrm{Stab}_{\GL(m)}(\mathcal{F}_\beta)\]
where 
\[\mathcal{F}_\beta=\{\langle e_i:\beta(i)\leq b_j\}_{j=1}^{q-1}.\]
In particular, if $d_j=\#\beta^{-1}(\{b_1,\ldots,b_j\})$ then $\GL(m)/Q$ can be identified with the flag variety $\Flags(d_1,\ldots,d_{q-1};W)$.

{\rm (b)} If $Q$ is a parabolic subgroup, the inclusion $U_Q\subset U_P$ of unipotent radicals holds if and only if any two distinct elements $(x_1,\ldots,x_d),(y_1,\ldots,y_d)$ of $\mathcal{I}$ satisfy $x_i\not=y_i$ for all $i\in\{1,\ldots,d\}$. 
\end{lemma}

\begin{proof}
(a) We have a decomposition
\[\gl(n)=\gl(V)=\mathfrak{h}\oplus\bigoplus_{1\leq i\not=j\leq n} \mathfrak{g}_{i,j}\]
where $\mathfrak{h}=\Lie\,H$ and $\mathfrak{g}_{i,j}=\mathbb{C}(e_i\otimes e_j^*)$. With this notation,
\begin{equation}
\label{p}
\mathfrak{p}:=\Lie\,P=\mathfrak{h}\oplus\bigoplus_{\alpha(i)\leq \alpha(j)}\mathfrak{g}_{i,j}\supset\nil(\mathfrak{p})=\bigoplus_{\alpha(i)<\alpha(j)}\mathfrak{g}_{i,j},
\end{equation}
where $\nil(\mathfrak{p})$ is the nilpotent radical of $\mathfrak{p}$.

There is a similar decomposition
\[\gl(m)=\gl(W)=\mathfrak{h}'\oplus\bigoplus_{1\leq i\not=j\leq m}\mathfrak{g}'_{i,j}.\]
Set $\mathfrak{q}:=\Lie\,Q$ where $Q=P\cap\GL(m)$ as before. Since we already know that $\mathfrak{h}'\subset\mathfrak{q}$, the subalgebra $\mathfrak{q}$ is parabolic if and only if
\begin{equation}
\label{cond-parabolic}
1\leq i\not=j\leq m\quad\Longrightarrow\quad (\mathfrak{g}'_{i,j}\subset\mathfrak{q}\quad\mbox{or}\quad\mathfrak{g}'_{j,i}\subset\mathfrak{q}). \end{equation}
In view of (\ref{p}) and the diagonal embedding $\gl(m)\subset\gl(n)$, whenever $1\leq i\not=j\leq m$ we have the equivalence
\begin{eqnarray*}
\mathfrak{g}'_{i,j}\subset\mathfrak{q}& \Longleftrightarrow& \mathfrak{g}_{i+km,j+km}\subset\mathfrak{p}\ \ \forall k=0,\ldots,d-1\\
&\Longleftrightarrow& \alpha(i+km)\leq \alpha(j+km)\ \forall k=0,\ldots,d-1 \\
&\Longleftrightarrow& \beta(i)\leq \beta(j).
\end{eqnarray*}
Hence, from (\ref{cond-parabolic}) we obtain that
$\mathfrak{q}$ is a parabolic subalgebra of $\gl(m)$ if and only if
\[1\leq i\not=j\leq m\quad\Longrightarrow\quad (\beta(i)\leq\beta(j)\quad\mbox{or}\quad\beta(j)\leq\beta(i)).\]
The condition means that $\leq$ is a total order set on $\mathcal{I}$.
We also have the equality
\begin{eqnarray*}
\mathfrak{q}=\mathfrak{h}\oplus\bigoplus_{\beta(i)\leq\beta(j)}\mathfrak{g}'_{i,j}&=&\{X\in\gl(W):X(\langle e_i:\beta(i)\leq b_j\rangle)\subset\langle e_i:\beta(i)\leq b_j\rangle\ \forall j\}\\&=&\Lie(\mathrm{Stab}_{\GL(m)}(\mathcal{F}_\beta))
\end{eqnarray*}
which implies that $Q=\mathrm{Stab}_{\GL(m)}(\mathcal{F}_\beta)$.

(b) Assume that $Q$ is a parabolic subgroup of $\GL(m)$. The inclusion $U_Q\subset U_P$ holds if and only if the similar inclusion holds for the nilradicals of the Lie algebras. Through the diagonal embedding of $\gl(m)$ into $\gl(n)$, the nilradical of $\mathfrak{q}$ can be described as
\[\nil(\mathfrak{q})=\bigoplus_{\substack{1\leq i\not=j\leq m \\\beta(i)<\beta(j)}}(\mathfrak{g}'_{i,j}\oplus
\mathfrak{g}'_{i+m,j+m}\oplus\ldots\oplus\mathfrak{g}'_{i+(d-1)m,j+(d-1)m}).\]
Therefore, the desired inclusion $\nil(\mathfrak{q})\subset\nil(\mathfrak{p})$ holds if and only if, for all $i,j\in\{1,\ldots,m\}$,
\[\beta(i)<\beta(j)\quad\Longleftrightarrow\quad \alpha(i+km)<\alpha(j+km)\ \ \forall k\in\{0,\ldots,d-1\}.\]
This condition is equivalent to the one stated in (b) (knowing that the partial order $\leq$ restricts to a total order on $\mathcal{I}$, due to (a)).
\end{proof}

\subsection{Diagonal embedding of flag varieties}
Assuming that the condition of Lemma \ref{L5.1}\,(a) is fulfilled, we now describe the embedding of partial flag varieties
\begin{equation}
\label{phi}
\phi:\GL(m)/Q=\Flags(d_1,\ldots,d_{q-1};W)\to \GL(n)/P
\end{equation}
obtained in this case. We rely on a combinatorial object, introduced in the next definition.

\begin{definition}
(a) We call \textit{E-graph} an unoriented graph
with the following features:
\begin{itemize}
\item The vertices consist of two sets $\{l_1,\ldots,l_q\}$ (``left vertices'') and $\{r_1,\ldots,r_p\}$ (``right vertices''), displayed from top to bottom in two columns, and two vertices are joined by an edge only if they belong to different sets. 
\item The edges display into $d$ subsets $E_c$ corresponding to a given colour $c\in\{1,\ldots,d\}$.
\item Every vertex is incident with at least one edge, and every vertex is incident with at most one edge of a given colour.
The vertex $l_q$ is incident with exactly $d$ edges (one per colour).
\item Two edges of the same colour never cross, that is, if $(l_i,r_j)$ and $(l_{k},r_{\ell})$, with $i<k$, are joined with two edges of the same colour, then $j<\ell$.
\end{itemize}
In an E-graph, we call ``bounding edges'' the edges passing through $l_q$, and we call ``ordinary edges'' all other edges.

(b) With the notation of Lemma \ref{L5.1}, we define the E-graph $\mathcal{G}(\alpha,\beta)$ such that
\begin{itemize}
\item we put an edge of colour $k$ between $l_i$ and $r_j$ whenever $b_i=(x_1,\ldots,x_d)$ satisfies $x_k=j$ and $i$ is maximal for this property.
\end{itemize}
(The conditions given in Lemma \ref{L5.1}\,(a) justify that $\mathcal{G}(\alpha,\beta)$ is a well-defined E-graph.)
\hfill $\blacksquare$
\end{definition}

In the following statement we describe explicitly the embedding $\phi$ of (\ref{phi}) and its properties in terms of the E-graph $\mathcal{G}(\alpha,\beta)$.

\begin{proposition}
\label{P-embeddings}
{\rm (a)} The map $\phi:Y=\GL(m)/Q=\Flags(d_1,\ldots,d_{q-1};W)\to X=\GL(n)/P$ is given by
\[\phi:\{F_1,\ldots,F_{q-1}\}\mapsto \{V_1,\ldots,V_{p-1}\}\]
where for all $j\in\{1,\ldots,p-1\}$ we have
\begin{equation}
\label{formula}
V_j=V_{j-1}+F_{i_1}^{(1)}\oplus\ldots\oplus F_{i_d}^{(d)},
\end{equation}
where $V_0=F_0:=0$, $F_q:=W$, and
$$i_k:=\left\{\begin{array}{ll}
i & \mbox{if the vertex $r_j$ is incident with an edge $(l_i,r_j)$ of colour $k$ in $\mathcal{G}(\alpha,\beta)$,} \\
0 & \mbox{if there is no edge of colours $k$ passing through $r_j$.}
\end{array}\right.
$$
We have also
\[V_j=F_{i'_1}^{(1)}\oplus\ldots\oplus F_{i'_d}^{(d)},\]
where
$i'_k$ is the index of the left end point of the last edge of colour $k$ arriving at or above $r_j$ in $\mathcal{G}(\alpha,\beta)$,
with $i'_k=0$ if there is no such edge.

{\rm (b)} Let $([L_1],\ldots,[L_{p-1}])$
and $([M_1],\ldots,[M_{q-1}])$ denote the sequences of preferred generators of $\Pic\,X$ and $\Pic\,Y$, respectively. The map $\phi^*:\Pic\,X\to\Pic\,Y$ is given by
\[\phi^*[L_j]=\sum_{k=1}^d[M_{i'_k}],\]
where we set by convention $[M_0]=[M_q]=0$.

{\rm (c)} The map $\phi$ is linear if and only if, whenever $r_j,r_{j'}$ with $j<j'$ are incident with edges of the same colour $c$ in the graph $\mathcal{G}(\alpha,\beta)$, every ordinary edge arriving at $r_{j''}$ for $j\leq j''<j'$ is also of colour $c$. 

{\rm (d)} The map $\phi$ is a standard extension if and only if all ordinary edges of $\mathcal{G}(\alpha,\beta)$ are of the same colour.
Moreover, in this case, $\phi$ is a strict standard extension.
\end{proposition}

\begin{proof}
(a) As in Section \ref{section-5.1}, we write $P=\mathrm{Stab}(\F_\alpha)$ where $\alpha:\{1,\ldots,n\}\to\{1,\ldots,p\}$ is surjective. 
Then we have $Q=\mathrm{Stab}(\mathcal{F}_\beta)$ where $\beta:\{1,\ldots,m\}\to \mathcal{I}\subset\{1,\ldots,p\}^d$ is described in Lemma \ref{L5.1}. 
Let $\hat\phi:\GL(m)/Q\to \GL(n)/P$ be the map given by formula (\ref{formula}).
Thus we have to show that $\hat\phi=\phi$.
Since the maps $\phi$ and $\hat\phi$ are $\GL(m)$-equivariant, it suffices to show that $\hat\phi(\F_\beta)=\F_\alpha$. 
We write $\F_\alpha=\{F_{\alpha,1},\ldots,F_{\alpha,p-1}\}$ and $\F_\beta=\{F_{\beta,1},\ldots,F_{\beta,q-1}\}$. For $j\in\{1,\ldots,p\}$, we have
\begin{equation}
\label{6-new}
F_{\alpha,j}=\langle e_i:\alpha(i)\leq j\rangle=F_{\alpha,j-1}+\langle e_i:\alpha(i)=j\rangle
\end{equation}
where $F_{\alpha,0}:=0$.
Every $i\in\{1,\ldots,n\}$ can be written $i=(k-1)m+r\in\{1,\ldots,n\}$ with $k\in\{1,\ldots,d\}$ and $r\in\{1,\ldots,m\}$, so that $e_i=\chi_k(e_r)$.

Assume that $\alpha(i)=j$. Then there is $b_{i'}=(x_1,\ldots,x_d)\in\mathcal{I}$ with $i'\in\{1,\ldots,q\}$ maximal such that $x_k=j$. Moreover there is $s\in\{r,\ldots,m\}$ such that $x_{\ell}=\alpha((\ell-1)m+s)$ for all $\ell\in\{1,\ldots,d\}$. This implies that the graph $\mathcal{G}(\alpha,\beta)$ contains an edge of colour $k$ joining $b_{i'}$ and $j$, and we have
$$
e_i=\chi_k(e_r)\in \chi_k(F_{\beta,i'})=F_{\beta,i'}^{(k)}
$$
where $F_{\beta,q}:=W$.
Conversely, assume that there is an edge of colour $k$ joining $b_{i'}$ and $j$. 
The subspace $F_{\beta,i'}^{(k)}$ is spanned by vectors of the form $\chi_k(e_r)$ with
 $r\in\{1,\ldots,m\}$  such that $\beta(r)=(\alpha((\ell-1)m+r)_{\ell=1}^d\leq b_{i'}$. The latter inequality implies  $\alpha((k-1)m+r)\leq j$. Hence $\chi_k(e_r)=e_{(k-1)m+r}\in F_{\alpha,j}$.

Combining these observations with (\ref{6-new}), we deduce that
$$
F_{\alpha,j}=F_{\alpha,j-1}+F_{\beta,i_1}^{(1)}\oplus\ldots\oplus F_{\beta,i_d}^{(d)}
$$
where $i_1,\ldots,i_d$ are as defined in the statement of the proposition.
Therefore, the claimed equality $\hat\phi(\mathcal{F}_\beta)=\mathcal{F}_\alpha$ holds.

The second formula stated in (a) is an immediate consequence of (\ref{formula}). The proof of (a) is complete.

Part (b) is a corollary of the second formula in (a), whereas parts (c) and (d) of the proposition easily follow from parts (a) and (b). The proof of the proposition is complete.
\end{proof}

\begin{remark}
\label{R-constant}
Proposition \ref{P-embeddings} shows how the E-graph $\mathcal{G}(\alpha,\beta)$ describes the embedding $\phi:Y\to X$.
Moreover, the chain of constant spaces $(C_j(\phi))$ is expressed in the following way. 
We enumerate the colours $k_1,\ldots,k_d$ so that $i_1\leq\ldots\leq i_d$ where $r_{i_j}$ is the right end point of the bounding edge of colour $k_j$. Then 
$$
C_j(\phi)=F_q^{(k_1)}\oplus\ldots\oplus F_q^{(k_j)} \quad\mbox{for $j=1,\ldots,d$.}
$$
\hfill $\blacksquare$
\end{remark}

\begin{example}
\label{E-standard-extension}
(a)
Let us consider for instance the graph
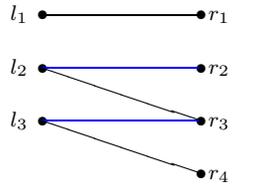
\begin{center}
\mbox{\tiny
$\begin{picture}(42,80)(0,0)
\put(-10,75){$l_1$}\put(0,75){$\bullet$}
\put(-10,55){$l_2$}\put(0,55){$\bullet$}
\put(-10,35){$l_3$}\put(0,35){$\bullet$}
\put(65,75){$r_1$}\put(60,75){$\bullet$}
\put(65,55){$r_2$}\put(60,55){$\bullet$}
\put(65,35){$r_3$}\put(60,35){$\bullet$}
\put(65,15){$r_{4}$}\put(60,15){$\bullet$}
\put(2,77){\line(1,0){59}}
\textcolor{blue}{\put(2,57){\line(1,0){59}}}
\put(2,57){\line(3,-1){59}}
\textcolor{blue}{\put(2,37){\line(1,0){59}}}
\put(2,37){\line(3,-1){59}}
\put(85,35){.}
\end{picture}$}
\end{center}
It encodes an embedding 
\begin{eqnarray*}
\phi:X=\Flags(d_1,d_2;\mathbb{C}^n) & \to &  Y=\Flags(d_1,d_1+d_2,d_2+n;\mathbb{C}^{2n}=\mathbb{C}^n\oplus\overline{\mathbb{C}^n})\\
\{V_1,V_2\}&\mapsto& \{V_1,V_1\oplus\overline{V_2},V_2\oplus\overline{\mathbb{C}^3}\}.
\end{eqnarray*}
If we denote by $([L_1],[L_2])$ and $([M_1],[M_2],[M_3])$  the sets of preferred generators of the Picard groups of $X$ and $Y$ respectively, then the induced map $\phi^*:\Pic\,Y\to\Pic\,X$ is given by
$$
[M_1]\mapsto[L_1],\quad [M_2]\mapsto [L_1]+[L_2],\quad [M_3]\mapsto [L_2].
$$
Thus $\phi$ is not linear in this case.

(b) Here we consider the graph
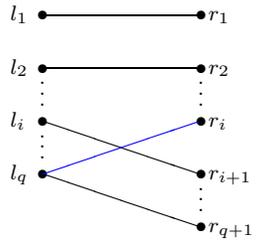
\begin{center}
\mbox{\tiny
$\begin{picture}(42,80)(0,0)
\put(-10,75){$l_1$}\put(0,75){$\bullet$}
\put(-10,55){$l_2$}\put(0,55){$\bullet$}
\put(1,43){$\vdots$}
\put(-10,35){$l_i$}\put(0,35){$\bullet$}
\put(1,23){$\vdots$}
\put(-10,15){$l_q$}\put(0,15){$\bullet$}
\put(65,75){$r_1$}\put(60,75){$\bullet$}
\put(65,55){$r_2$}\put(60,55){$\bullet$}
\put(61,43){$\vdots$}
\put(65,35){$r_i$}\put(60,35){$\bullet$}
\put(61,03){$\vdots$}
\put(65,15){$r_{i+1}$}\put(60,15){$\bullet$}
\put(65,-5){$r_{q+1}$}\put(60,-5){$\bullet$}
\put(2,77){\line(1,0){59}}
\put(2,57){\line(1,0){59}}
\put(2,37){\line(3,-1){59}}
\put(2,17){\line(3,-1){59}}
\textcolor{blue}{\put(2,17){\line(3,1){59}}}
\put(85,35){.}
\end{picture}$}
\end{center}
There are two colours which means that the embedding is from a flag variety of a space $V$ to the flag variety of a doubled space $W=V\oplus\overline{V}$:
\[\mathrm{Fl}(d_1,\ldots,d_{q-1};V)\hookrightarrow \mathrm{Fl}(d'_1,\ldots,d'_{q};W=V\oplus\overline{V}).\]
The embedding has the following explicit form
\begin{equation}
\label{se1}
\{F_1,\ldots,F_{q-1}\}\mapsto\{F_1,\ldots,F_{i-1},F_{i-1}\oplus\overline{V},\ldots,F_{q-1}\oplus\overline{V}\}.
\end{equation}
Note that $\dim F_{i-1}\oplus\overline{V}/\dim F_{i-1}=\dim V$. The dimensions of the other quotients are unchanged.

(c) Now consider
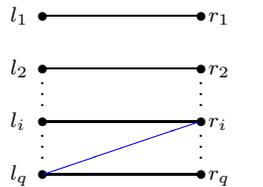
\begin{center}
\mbox{\tiny
$\begin{picture}(42,80)(0,0)
\put(-10,75){$l_1$}\put(0,75){$\bullet$}
\put(-10,55){$l_2$}\put(0,55){$\bullet$}
\put(1,43){$\vdots$}
\put(-10,35){$l_i$}\put(0,35){$\bullet$}
\put(1,23){$\vdots$}
\put(-10,15){$l_q$}\put(0,15){$\bullet$}
\put(65,75){$r_1$}\put(60,75){$\bullet$}
\put(65,55){$r_2$}\put(60,55){$\bullet$}
\put(61,43){$\vdots$}
\put(65,35){$r_i$}\put(60,35){$\bullet$}
\put(61,23){$\vdots$}
\put(65,15){$r_{q}$}\put(60,15){$\bullet$}
\put(2,77){\line(1,0){59}}
\put(2,57){\line(1,0){59}}
\put(2,37){\line(1,0){59}}
\put(2,17){\line(1,0){59}}
\textcolor{blue}{\put(2,17){\line(3,1){59}}}
\put(85,35){.}
\end{picture}$}
\end{center}
In this case we get an embedding\[\mathrm{Fl}(d_1,\ldots,d_q;V)\hookrightarrow \mathrm{Fl}(d'_1,\ldots,d'_q;W=V\oplus\overline{V})\]
given by
\begin{equation}
\label{se2}
\{F_1,\ldots,F_{q-1}\}\mapsto\{F_1,\ldots,F_{i-1},F_{i}\oplus\overline{V},\ldots,F_{q-1}\oplus\overline{V}\}.
\end{equation}
The only quotient whose dimension changes is $F_i\oplus\overline{V}/F_{i-1}$ which has dimension $\dim V+\dim F_i/F_{i-1}$.

By Proposition \ref{P-embeddings}\,(d), the embeddings of parts (a) and (b) of this example are the only possible standard extensions that can come from a diagonal embedding $\GL(n)\hookrightarrow\GL(2n)$.

(d) In the case of a diagonal embedding of the form $\GL(n)\hookrightarrow \GL(dn)$, if the embedding of flag varieties is a standard extension, then it can be described as a composition of embeddings of the previous form, involving a subspace $\overline{V}$ still of dimension $n$.
\hfill $\blacksquare$
\end{example}

\begin{remark}
The fact that $U_Q\subset U_P$ is equivalent to the following property of the graph $\mathcal{G}(\alpha,\beta)$: every left vertex is incident with exactly $d$ edges (one per colour).
\hfill $\blacksquare$
\end{remark}

Proposition \ref{P-embeddings} has the following corollary.

\begin{corollary}
\label{L-embed-Picard}
For an embedding $\phi:Y=\GL(n)/Q\to X=\GL(m)/P$ as in Proposition \ref{P-embeddings} and 
for every $j\in\{1,\ldots,q-1\}$, 
we have $\mathrm{Im}\,\phi^*\not\subset\langle [M_i]:i\in\{1,\ldots,q-1\}\setminus\{j\}\rangle$.
\end{corollary}





\subsection{Application to ind-varieties}

\label{s-graphs}

\begin{definition}
Let $\{s_n\}_{n\geq 1}$ be an exhaustion of $\sn$.
We call {\em $\sn$-graph} a graph with infinitely many columns of vertices $B_n$, with $1\leq |B_n|\leq s_n$ for all $n\geq 1$, such that the subgraph consisting of $B_n,B_{n+1}$ and the corresponding edges is an E-graph. 
\hfill $\blacksquare$
\end{definition}

A parabolic subgroup $\mathbf{P}$ of $\GL(\sn)$ gives rise to an $\sn$-graph. According to the above proposition, this graph encodes the embeddings of flag varieties in an exhaustion of $\GL(\sn)/\mathbf{P}$.
Conversely, any $\sn$-graph arises from a parabolic subgroup $\mathbf{P}$ of $\GL(\sn)$.


\section{Ind-varieties of generalized flags as homogeneous spaces of $\GL(\sn)$}

\label{section-6}

Our purpose in this section is to characterize ind-varieties of generalized flags (introduced in Section \ref{section-4.2}) which can be realized as homogeneous spaces $\GL(\sn)/\mathbf{P}$ for the given supernatural number $\sn$.

\subsection{The case of finitely many finite-dimensional subspaces}

We start with a special situation which is easier to deal with: let $\mathbf{X}=\Flags(\F,E)$ where $\mathcal{F}=\{F'_a,F''_a\}_{a\in A}$ is an $E$-compatible generalized flag, for an arbitrary totally ordered set $(A,\leq)$, but with the assumption that
\begin{equation}
\label{9}
\dim F''_a/F'_a=+\infty\ \ \mbox{for all but finitely many $a\in A$.}
\end{equation} 

\begin{theorem}
\label{T-6.1}
If condition (\ref{9}) holds, then for every supernatural number $\sn$, there is an isomorphism of ind-varieties $\Flags(\mathcal{F},E)\cong\GL(\sn)/\mathbf{P}$ for an appropriate parabolic subgroup $\mathbf{P}\subset\GL(\sn)$.
\end{theorem}

\begin{proof}


In the situation of the theorem, the ind-variety $\mathbf{X}=\Flags(\mathcal{F},E)$ has an exhaustion 
$$
X_1\hookrightarrow X_2\hookrightarrow\cdots\hookrightarrow X_n\stackrel{\phi_n}{\hookrightarrow} X_{n+1}\hookrightarrow\cdots
$$
such that $X_n$ is a finite-dimensional variety of flags in the space $\mathbb{C}^{s_n}$ for some exhaustion $\{s_n\}_{n\geq 1}$ of $\sn$ with $s_1$ sufficiently large,
and $\phi_n:X_n\to X_{n+1}$ is one of the two maps from Example \ref{E-standard-extension} (b) and (c). 
Using the maps $\phi_n$, one constructs nested parabolic subgroups $P_n\subset\GL(s_n)$
such that
 $X_n\cong\GL(s_n)/P_n$ and $P_n=\GL(s_n)\cap P_{n+1}$ for all $n$. The union $\mathbf{P}=\bigcup_{n\geq 1}P_n$ is then a parabolic subgroup of $\GL(\sn)$ which satisfies the conditions of the theorem.
\end{proof}

\subsection{The general case} To treat the general case, we need to start with a definition.

\begin{definition}
Let $\mathcal{F}=\{F'_a,F''_a\}_{a\in A}$ be an $E$-compatible generalized flag and let $A'=\{a\in A:\dim F''_a/\dim F'_a<+\infty\}$. We say that
the ind-variety $\mathrm{Fl}(\mathcal{F},E)$ is {\em $\sn$-admissible}
if either $A'$ is finite or $A'$ is infinite and there are a exhaustion $\{s_n\}_{n\geq 1}$ for $\sn$ and a numbering $A'=\{k_n\}_{n\geq 1}$ (not necessarily compatible with the total order on $A'$) such that, for all $n\geq 0$:
\begin{center}
$\frac{\dim F''_{k_n}/F'_{k_n}}{s_n}\in\{1,\ldots,\frac{s_{n+1}}{s_n}-1\}$ and $s_n|\dim F''_a/F'_a$ for all $a\in A'\setminus\{k_1,\ldots,k_n\}$.\end{center}
\hfill $\blacksquare$
\end{definition}

\begin{theorem}
\label{T-6.3}
The following conditions are equivalent:
\begin{itemize}
\item[\rm (i)] The ind-variety $\Flags(\mathcal{F},E)$ is $\sn$-admissible.
\item[\rm (ii)] There is a parabolic subgroup $\mathbf{P}\subset\GL(\sn)$ and an isomorphism of ind-varieties $\Flags(\mathcal{F},E)\cong\GL(\sn)/\mathbf{P}$.
\end{itemize}
\end{theorem}

\begin{proof}
 (i)$\Rightarrow$(ii): 
The ind-variety $\mathrm{Fl}(\mathcal{F},E)$ admits an exhaustion $\Flags(\mathcal{F},E)=\bigcup_nX_n$
with embeddings of the form
\begin{eqnarray}
\label{9-newnew}
\phi_n:X_n=\mathrm{Fl}(p_1,\ldots,p_{k_n};V_n) &\to & X_{n+1}=\mathrm{Fl}(q_{1},\ldots,q_{\ell_n};V_n\oplus C_n),\\ \{F_1,\ldots,F_{k_n}\} & \mapsto & \{F_{\tau(1)}\oplus C_1^n,\ldots,F_{\tau(\ell_n)}\oplus C^n_{\ell_n}\}\nonumber
\end{eqnarray}
(with $F_0:=0$, $F_{k_n+1}:=V_n$) for a nondecreasing surjective map $\tau:\{1,\ldots,\ell_n\}\to\{0,1,\ldots,k_n,k_n+1\}$ and a sequence $C^n_1\subset\ldots\subset C^n_{\ell_n}$ (with possible repetitions) of subspaces of $C_n$.

Assume that there is another exhaustion $\mathrm{Fl}(\mathcal{F},E)=\bigcup_nY_n$ for which the embeddings are as described in Proposition \ref{P-embeddings},
where $Y_n=\mathrm{Fl}(r_1,\ldots,r_{m_n};W_n)$ and $\dim W_n=s_n$ for an exhaustion $\{s_n\}$ of $\sn$. Then the two exhaustions interlace,
and there is no loss of generality in assuming that the interlacing holds for the sequences $(X_n)$ and $(Y_n)$, and not only for subsequences:
\[\xymatrix{X_n \ar@{^{(}->}[r]^{\phi_n} \ar@{^{(}->}[d]^{\xi_n}
& X_{n+1} \ar@{^{(}->}[d]^{\xi_{n+1}} \\
Y_n \ar@{^{(}->}[r]^{\psi_n} \ar@{^{(}->}[ur]^{\chi_n} & Y_{n+1}.} \]

\medskip

\noindent
{\it Claim.} The embedding $\xi_n$ is a standard extension.

First we show that $\xi_n$ is linear.
Arguing by contradiction, assume that there is a generator $[M_i]$ among the sequence 
$[M_1],\ldots,[M_{q-1}]$
of preferred generators of $\Pic\,X_n$ such that $\xi_n^*[M_i]$ is neither $0$ nor a preferred generator of $\Pic\,Y_n$. Since $\phi_n^*=\chi_n^*\circ\xi_{n+1}^*$, we have the inclusion
$\mathrm{Im}\,\phi_n^*\subset\mathrm{Im}\,\chi_n^*$, and due to Corollary \ref{L-embed-Picard}
we get that there is a generator $[L]\in\Pic\,Y_{n+1}$
such that 
\[\chi_n^*[L]=\sum_{j=1}^{q-1}\lambda_j [M_j]\quad\mbox{with $\lambda_i\not=0$.}\]
Since the map $\chi_n$ is an embedding, we have $\lambda_j\geq 0$ for all $j$ and in particular $\lambda_i\geq 1$.
The same argument applied  to $\xi_n$ implies that $\xi_n^*[M_j]$ should be a linear combination of the
preferred generators of $\Pic\,Y_n$ with nonnegative integer coefficients.
This implies that $\psi_n^*[L]=\xi_n^*\chi_n^*[L]$ is neither $0$ nor a preferred generator of $\Pic\,Y_n$, contradicting the linearity of the standard extension $\psi_n$.

Recall that in \cite{PT} the notion of an embedding factoring through a direct product is introduced. Note that $\xi_n$ cannot factor through a direct product: 
otherwise, $\psi_n$ would also factor through a direct product, 
which is impossible since this is a standard extension. Consequently, $\xi_n$ is a standard extension, and the claim is established.

\medskip
Now we can assume that $\xi_1$ is a strict standard extension. Since the maps $\phi_n$ are strict standard extensions, by using the formula for $\psi_n$ in Proposition \ref{P-embeddings} we derive that $\xi_n$ is a strict standard extension for all $n\geq 1$.

\medskip
Due to (\ref{9-newnew}) and Proposition \ref{P-embeddings}, one  has $W_n=V_n\oplus Z_n$ and the map $\xi_n$ has the form
$$
\xi_n:\{F_1,\ldots,F_{k_n}\}\mapsto
\{F_{\sigma(1)}\oplus Z_1^n,\ldots,F_{\sigma(p_n)}\oplus Z^n_{p_n}\}.
$$
Since this applies likewise to $\xi_{n+1}$, and taking into account the form of $\phi_n$ in (\ref{9-newnew}), we see that the map $\psi_n$ has the form
$$
\{F_1,\ldots,F_{k_n},F'_1,\ldots,F'_{p_n}\}\mapsto 
\{F_1,\ldots,F_{k_n},R^n_1,\ldots,R^n_{\ell_n},\Gamma^n_1,\ldots,\Gamma^n_{\delta_n}\}
$$
for an arbitrary map $\zeta_n:\{F'_1,\ldots,F'_{p_n}\}\mapsto \{R^n_1,\ldots,R^n_{\ell_n}\}$
as described in Proposition \ref{P-embeddings}, and where  $\Gamma^n_i$ are constant subspaces which are copies of $W_n$ in $W_{n+1}=\bigoplus_{i=1}^{d_n}W_n^{(i)}$.
This implies  $\dim V_n=d'_ns_n$ for some $d'_n\in\{1,\ldots,d_n=\frac{s_{n+1}}{s_n}\}$.

Since $\{V_n\oplus Z^n_1,\ldots,V_n\oplus Z^n_{p_n}\}=\{\Gamma^{n-1}_1,\ldots,\Gamma^{n-1}_{\delta_{n-1}}\}$,
we must have $p_n=\delta_{n-1}$
and the dimension of $Z_i^n$ is a multiple of $s_{n-1}$.
Therefore $\dim R_i^n$ is also a multiple of $s_{n-1}$ for all $i$.
Condition (ii) is established.

\medskip
(ii)$\Rightarrow$(i):
Let $d'_n=\frac{\dim F_{k_n}}{s_n}\in\{1,\ldots,d_n\}$ and set $V_n=F_{k_n}$. The conditions imply that we can choose a decomposition $W_n=V_n\oplus W_{n-1}^{(1)}\oplus\ldots\oplus W_{n-1}^{(d_n-d'_n)}$ where the $W_{n-1}^{(i)}$'s are copies of $W_{n-1}$ such that the strict standard extension
$\mathrm{Fl}_n(\mathcal{F},E)\to\mathrm{Fl}_{n+1}(\mathcal{F},E)$ is given by
$$
\phi_n:\{F_1,\ldots,F_{k_n}\}\mapsto \{F_{k_1}+C^n_1,\ldots,F_{k_{n+1}}+C^n_{k_{n+1}}\}
$$
with $C^n_i=W_{n-1}^{(1)}\oplus\ldots\oplus W_{n-1}^{(m_i)}=C_i'^n\oplus C_i''^n$ for some nondecreasing sequence $m_1,\ldots,m_{k_{n+1}}$.

Letting $\xi_n:\Flags_n(\mathcal{F},E)\to \Flags(\mathbf{t}_n;W_n)$ be given by
$$
\xi_n:\{F_1,\ldots,F_{k_n}\}\mapsto \{F_{k_1}+C'^n_1,\ldots,F_{k_{n+1}}+C'^n_{k_{n+1}}\},
$$
and $\psi_n:\Flags(\mathbf{t}_n;W_n)\to\Flags(\mathbf{t}_{n+1};W_{n+1})$
$$
\psi_n:
\{F_1,\ldots,F_{k_n}\}\mapsto 
\{F_1+ C''^{n+1}_1,\ldots,F_{\ell_n+1}+ C''^{n+1}_{\ell_{n+1}}\}
$$
(for suitable types $\mathbf{t}_n$), we get exhaustions of $\Flags(\mathcal{F},E)$ and a homogeneous space for $\GL(\sn)$, which interlace.
Hence if $\mathcal{F}$ satisfies the condition above, then we can realize $\mathrm{Fl}(\mathcal{F},E)$ as a homogeneous space for $\GL(\sn)$.
\end{proof}

\begin{remark}
It is shown in \cite[Corollary 5.40]{DP1} that $\GL(\sn)/\mathbf{B}$ is never projective when $\mathbf{B}$ is a Borel subgroup. 
On the other hand, according to \cite[Proposition 7.2]{Dimitrov-Penkov}, an ind-variety of generalized flags is projective if and only if
the total order on the flag can be induced by a subset of $(\mathbb{Z},\leq)$,
and Theorem \ref{T-6.3} shows that in many  situations  $\GL(\sn)/\mathbf{P}$ is projective.~\hfill $\blacksquare$
\end{remark}

\section{The case of direct products of ind-varieties of generalized flags}

\label{section-7}

In this section, we point out that many direct products of ind-varieties of generalized flags can be homogeneous spaces for the group $\GL(\sn)$.

\subsection{Direct products of ind-varieties}

Let $\mathbf{X}_i=\bigcup_{n\geq 1}X_{i,n}$ ($i\in I$) be a collection of ind-varieties indexed by $\mathbb{Z}_{>0}$ or a finite subset of it. 
For each $i\in I$ we pick an element $x_i\in X_{i,1}$ and we set $X_{i,0}=\{x_i\}$.
The direct product in the category of pointed ind-varieties is then given by 
$$\prod_{i\in I}\mathbf{X}_i:=\bigcup_{n\geq 1}\prod_{i\in I} X_{i,\phi_i(n)}$$ for a collection of increasing maps $\phi_i:\mathbb{Z}_{>0}\to\mathbb{Z}_{\geq 0}$ such that for every $n\in\mathbb{Z}_{>0}$ we have $\phi_i(n)=0$ for all but finitely many $i\in I$ (the definition does not depend essentially on the choice of the maps $\phi_i$).

\begin{remark}
(a) As a set, the direct product can be identified with the set of sequences $(y_i)_{i\in I}$ where $y_i\in \mathbf{X}_i$ for all $i\in I$ and $y_i=x_i$ for all but finitely many $i\in I$.

(b) For a finite set of indices $I$, as a set, $\prod_{i\in I}\mathbf{X}_i$ coincides with the usual cartesian product, and its structure of ind-variety is given by the exhaustion $\prod_{i\in I}\mathbf{X}_i:=\bigcup_{n\geq 1}X_{i,n}$.
\end{remark}

Fixing an index $i_0\in I$, there are a canonical projection
$$
\mathrm{proj}_{i_0}:\prod_{i\in I}\mathbf{X}_i\to \mathbf{X}_{i_0},\ (y_i)\mapsto y_{i_0}$$
and an embedding
$$\mathrm{emb}_{i_0}:\mathbf{X}_{i_0}\to \prod_{i\in I}\mathbf{X}_i,\ x\mapsto (y_i)\ \mbox{with}\ y_i=\left\{
\begin{array}{ll}
x_i & \mbox{if $i\not=i_0$}, \\ x & \mbox{if $i=i_0$,}
\end{array}
\right.
$$
which are morphisms of ind-varieties. 

If the product is endowed with an action of a group $G$, then each ind-variety $\mathbf{X}_i$ inherits an action of $G$ defined through the maps $\mathrm{proj}_{i}$ and $\mathrm{emb}_i$.
Conversely, if every ind-variety $\mathbf{X}_i$ is endowed with an action of a group $G$, then we obtain an action of $G$ on the product defined diagonally provided that the following condition is fulfilled:
\begin{equation}
\label{condition}
\mbox{every $g\in G$ fixes $x_i$ for all but finitely many $i\in I$.}
\end{equation}
(This condition is automatically satisfied in the case where $I$ is finite.) Moreover, in both directions, when $G=\mathbf{G}$ is an ind-group, we have that the obtained action is algebraic provided that the initial one is.
The following lemma is an immediate consequence of this discussion.

\begin{lemma}
Assume that the direct product $\prod_{i\in I}\mathbf{X}_i$ is a homogeneous space for an ind-group $\mathbf{G}$. Then, every ind-variety $\mathbf{X}_i$ is also a homogeneous space for $\mathbf{G}$.
\end{lemma}

Note also that a direct product $\prod_{i\in I}\mathbf{X}_i$ is locally projective if and only if it is the case of $\mathbf{X}_i$ for all $i\in I$.

\subsection{The case of ind-varieties of generalized flags}

We start with an example.

\begin{example}
Let $\sn=2^\infty$. We consider the space $V$ of countable dimension, endowed with its fixed basis $E=\{e_k\}_{k\in\mathbb{Z}_{>0}}$, and we set $V_n:=\langle e_1,\ldots,e_n\rangle$. We have the exhaustion $\GL(\sn)=\bigcup_{n\geq 1}\GL(V_{2^n})$ defined through the diagonal embedding $\GL(V_{2^n})\hookrightarrow \GL(V_{2^{n+1}})$, $x\mapsto\mbox{\scriptsize$\begin{pmatrix} x & 0 \\ 0 & x \end{pmatrix}$}$. Consider the sequence of parabolic subgroups
$$
P_n:=\mathrm{Stab}_{\GL(V_{2^n})}(V_1,V_{2^{n}-1}),\quad n\geq 2.
$$
In this way, $P_n\cap\GL(V_{2^{n-1}})=P_{n-1}$ for all $n\geq 3$. Moreover,
every quotient $\GL(V_{2^n})/P_n$ is a flag variety formed by flags $(F_1\subset F_2\subset V_{2^n})$ of length 2, and we have an embedding of flag varieties
$$
\GL(V_{2^{n-1}})/P_{n-1}\to\GL(V_{2^n})/P_n,\ (F_1,F_2)\mapsto (F_1,V_{2^{n-1}}+\overline{F_2})
$$
where as before the map
$$V_{2^{n-1}}=\langle e_1,\ldots,e_{2^{n-1}}\rangle\to \overline{V_{2^n-1}}=\langle e_{2^{n-1}+1},\ldots,e_{2^{n-1}+2^{n-1}}\rangle,\ v\mapsto\overline{v},$$ is the  isomorphism from Section \ref{section-5}.
For every $n$, 
this embedding factors through a direct product of grassmannians 
$$
\GL(V_{2^{n-1}})/P_{n-1} \to \Gr(1;V_{2^{n-1}})\times \Gr(2^{n-1}-1;\overline{V_{2^{n-1}}}) \to \GL(V_{2^n})/P_n,
$$
which allows us to chech that the ind-variety $\GL(\sn)/\mathbf{P}$, where $\mathbf{P}=\bigcup_n P_n$, is isomorphic as an ind-variety to a direct product of two ind-grassmannians.
\end{example}


The following theorem shows that many homogeneous spaces for $\GL(\sn)$ can be isomorphic to direct products of ind-varieties of generalized flags.

\begin{theorem}
\label{T-7.4}
Let $\mathbf{X}=\GL(\sn)/\mathbf{P}$ be a homogeneous space, defined by a parabolic subgroup $\mathbf{P}$. Assume that we have an exhaustion $\mathbf{X}=\bigcup_{n\geq 1}\GL(s_n)/P_{s_n}$ determined by an exhaustion $\{s_n\}_{n\geq 1}$ of $\sn$, where each embedding $\GL(s_n)/P_{s_n}\hookrightarrow \GL(s_{n+1})/P_{s_{n+1}}$ is linear. Then, $\mathbf{X}$ is isomorphic as an ind-variety to a direct product of ind-varieties of generalized flags $\prod_{i\in I}\Flags(\mathcal{F}^i,E^i)$
where $I$ is either $\mathbb{Z}_{>0}$ or a finite subset of it.
\end{theorem}

\begin{proof}
Fix $n\geq 1$. 
Let $d_n=\frac{s_{n+1}}{s_n}$ and fix a decomposition
\begin{equation}
\label{dec}
V_n=W_n^{(1)}\oplus\ldots\oplus W_n^{(d_n)}
\end{equation} of the space $V_n=\mathbb{C}^{s_{n+1}}$ as in Section \ref{section-5}.
Thus we are in the setting of Proposition \ref{P-embeddings}, and the embedding
$$
\phi_n:\GL(s_n)/P_{s_n}\hookrightarrow \GL(s_{n+1})/P_{s_{n+1}}
$$
can be encoded by an E-graph with $d_n$ colours in the sense of Proposition \ref{P-embeddings}\,(a). The formula therein, combined with the characterization of $\phi_n$ given in Proposition \ref{P-embeddings}\,(c), yields a commutative diagram
$$
\xymatrix{\GL(s_n)/P_{s_n} \ar[r]^{\phi_n} \ar[d]^{\prod_i\psi^{(i)}} & \GL(s_{n+1})/P_{s_{n+1}} \\ \prod_{i=1}^{d_n}\Flags(\underline{\ell}^{(i)};W_n^{(i)})
\ar[ru]^{\xi_n}}
$$
where $\psi^{(i)}$ is the embedding corresponding to the subgraph of $\mathcal{G}(\alpha,\beta)$ formed by removing all the ordinary edges which are not of colour $i$ ($\psi^{(i)}$ is a strict standard extension due to Proposition \ref{P-embeddings}\,(c)--(d)), $\underline{\ell}^{(i)}$ is an appropriate dimension vector, and the embedding $\xi_n$ is induced by the decomposition (\ref{dec}).
The theorem follows from this construction.
\end{proof}

The above proof yields the following sharpening of Theorem \ref{T-7.4}.

\begin{corollary}
In the framework of Theorem \ref{T-7.4}, let $\mathcal{G}$ be the $\sn$-graph corresponding to $\mathbf{P}$ in the sense of Section \ref{s-graphs}. Let $\mathcal{G}=\bigcup_{i\in\mathcal{I}}\mathcal{G}_i$ be a decomposition into subgraphs so that all ordinary edges of $\mathcal{G}_i$ are of the same colour. Then, 
the ind-variety $\mathbf{X}$ is isomorphic to a direct product of ind-varieties of generalized flags $\prod_{i\in\mathcal{I}}\mathbf{X}_i$ where $\mathbf{X}_i$ has an exhaustion with embeddings encoded by  $\mathcal{G}_i$.
\end{corollary}


%

%

\section*{Outlook}

We see the results of this paper as a small first step in the study of locally projective homogeneous ind-spaces of locally reductive ind-groups. One inevitable question for a future such study is, given two non-isomorphic locally reductive ind-groups $\mathbf{G}$ and $\mathbf{G}'$, when are two homogeneous spaces $\mathbf{G}/\mathbf{P}$ and $\mathbf{G}'/\mathbf{Q}$ isomorphic as ind-varieties ? A further natural direction of research could be a comparison of Bott--Borel--Weil type results on $\mathbf{G}/\mathbf{P}$ and $\mathbf{G}'/\mathbf{Q}$. We finish the paper by pointing out that the reader can verify that Theorem \ref{T-6.1} remains valid if one replaces $\GL(\sn)$ by any pure diagonal ind-group in the terminology of \cite{BZ}.





\end{document}